\documentclass[oneside,12pt]{amsart}
\usepackage[T2A]{fontenc}
\usepackage[cp1251]{inputenc}
\usepackage{amssymb}
\usepackage{amsxtra}
\usepackage{amsmath}
\usepackage{amsthm}
\usepackage{hyperref}
\usepackage[all]{xy}

\usepackage[russian,english]{babel}

\DeclareMathOperator{\GL}{GL}
\DeclareMathOperator{\EE}{E}

\DeclareMathOperator{\Cent}{Cent}

\DeclareMathOperator{\Gm}{{\mathbf G}_m}

\DeclareMathOperator{\SL}{SL}
\DeclareMathOperator{\Sp}{Sp}

\DeclareMathOperator{\SO}{SO}

\newcommand{\ad}{\mathrm{ad}}
\newcommand{\scl}{\mathrm{sc}}

\DeclareMathOperator{\ZZ}{{\mathbb Z}}

\newcommand{\Aff}{\mathbb {A}}

\newtheorem{lem}{Lemma}[section]
\newtheorem*{thm*}{Theorem}
\newtheorem{thm}[lem]{Theorem}

\newtheorem{cor}[lem]{Corollary}
\newtheorem*{cor*}{Corollary}
\theoremstyle{definition}{    }
\theoremstyle{definition}{   }
\theoremstyle{definition}{   }

\newcommand{\fppf}{\mathrm{fppf}}

\begin{document}

\title{Chevalley groups of polynomial rings over Dedekind domains}
\author{A. Stavrova}
\thanks{The author is a winner of the contest ``Young Russian Mathematics''. The work was supported by the RFBR grant 18-31-20044}
\address{Chebyshev Laboratory, St. Petersburg State University,
14th Line V.O. 29B, 199178 Saint Petersburg, Russia}
\email{anastasia.stavrova@gmail.com}

\maketitle

\begin{abstract}
Let $R$ be a Dedekind domain, and let $G$ be a split reductive group, i.e.
a Chevalley--Demazure group scheme, of rank $\ge 2$.
We prove that $G(R[x_1,\ldots,x_n])=G(R)E(R[x_1,\ldots,x_n])$ for any $n\ge 1$.
This extends the corresponding
results of A. Suslin
and F. Grunewald, J. Mennicke, and L. Vaserstein for $G=\SL_N$, $\Sp_{2N}$.
We also deduce some corollaries of the above result for regular rings $R$ of higher dimension and discrete Hodge algebras
over $R$.
\end{abstract}

\section{Introduction}

A. Suslin~\cite[Corollary 6.5]{Sus} established that for any regular ring $R$ of dimension $\le 1$, any $N\ge 3$, and any $n\ge 1$,
one has
$$
\SL_N(R[x_1,\ldots,x_n])=\SL_N(R)E_N(R[x_1,\ldots,x_n]),
$$
 where $E_N(R[x_1,\ldots,x_n])$ is the elementary
subgroup, i.e. the subgroup generated by elementary matrices $I+te_{ij}$, $1\le i\neq j\le N$, $t\in R[x_1,\ldots,x_n]$.
In particular, this implies
$$
\SL_N(\ZZ[x_1,\ldots,x_n])=E_N(\ZZ[x_1,\ldots,x_n]).
$$
A later theorem of A. Suslin and V. Kopeiko~\cite[Theorem 7.8]{Sus-K-O1} together with the homotopy invariance
of orthogonal $K$-theory (see~\cite[Corollaire 0.8]{Karoubi-per},~\cite[Corollary 1.12]{Ho-A1repr}, or~\cite[Theorem 9.8]{Schl-hermit}) implies
a similar result for even orthogonal
groups $SO_{2N}$, $N\ge 3$, under the additional assumption $2\in R^\times$.
F. Grunewald, J. Mennicke, and L. Vaserstein~\cite{MenGruVas} extended the result of Suslin to symplectic groups
$\Sp_{2N}$, $N\ge 2$, and a slightly larger class
of rings $R$, namely, locally principal ideal rings.
One says that a (commutative associative) ring $A$ with 1 is a locally principal ideal ring,
if for every maximal ideal $m$ of $A$ the localization $A_m$ is a principal ideal
ring.

Our aim is to extend the above results to all Chevalley--Demazure group schemes of isotropic rank $\ge 2$.
By a Chevalley--Demazure group scheme we mean a split reductive group scheme in the sense of~\cite{SGA3}. These group
schemes are defined over $\ZZ$. We say that a Chevalley--Demazure
group scheme $G$ has isotropic rank $\ge n$ if and only if every irreducible component of its root system has rank $\ge n$.
For any commutative ring $R$ with 1 and any fixed choice of a pinning, or \'epinglage of $G$ in the sense of~\cite{SGA3},
we denote by $E$ the elementary subgroup functor of $G$. That is, $E(R)$ is the subgroup of $G(R)$ generated
by elementary root unipotent elements $x_\alpha(r)$, $\alpha\in\Phi$, $r\in R$, in the notation of~\cite{Che55,Ma},
where $\Phi$ is the root system of $G$. If $G$
has isotropic rank $\ge 2$, then $E$ is independent of the choice of the pinning~\cite{PS}.

Our main result is the following theorem. Since
$\SL_2(\ZZ[x])\neq E_2(\ZZ[x])$~\cite{Cohn}, it cannot be extended to the case of
isotropic rank 1.

\begin{thm}[Theorem~\ref{thm:main-2}]\label{thm:main}
Let $R$ be a locally principal ideal ring, and let $G$ be a simply connected Chevalley--Demazure group scheme of isotropic rank $\ge 2$.
Then $G(R[x_1,\ldots,x_n])=G(R)E(R[x_1,\ldots,x_n])$ for any $n\ge 1$.
\end{thm}

Theorem~\ref{thm:main} for Dedekind domains
was previously claimed by M. Wendt~\cite[Proposition 4.7]{W10}, however, his proof was incorrect~\cite[p. 91]{Stepanov-elem}. We give another proof
along the lines similar to~\cite{MenGruVas}. The case where $R$ is a field was done earlier in~\cite{St-poly} in a more
general context of isotropic reductive groups.

Following~\cite{Ma}, we say that a Dedekind domain $R$ is of arithmetic type, if $R=O_S$
is the ring of $S$-integers of a global field $k$ with respect to a finite non-empty set
$S$ of primes containing all archimedean primes.

\begin{cor}\label{cor:ari}
Let $R$ be a Dedekind domain of arithmetic type (e.g. $R=\ZZ$),
 and let $G$ be a simply connected Chevalley--Demazure group scheme of isotropic rank $\ge 2$.
Then $G(R[x_1,\ldots,x_n])=E(R[x_1,\ldots,x_n])$ for any $n\ge 1$.
\end{cor}
\begin{proof}
This follows from Theorem~\ref{thm:main} and~\cite[Th\'eor\`eme 12.7]{Ma}, which says that $G(R)=E(R)$.
\end{proof}

Note that~\cite[p. 57]{Lam-book06} presents an example (due to J. Stallings)
of a Dedekind domain $D$ such that $\SL_3(D)\neq E_3(D)$, hence Corollary~\ref{cor:ari} does not hold for arbitrary
Dedekind domains.

A commutative $R$-algebra of the form $A=R[x_1,\ldots,x_n]/I$, where
$I$ is an ideal generated by monomials, is called a discrete Hodge algebra over $R$.
If $I$ is generated by square-free monomials,
 $A$ is called a square-free discrete Hodge algebra. The simplest example of such an algebra is $R[x,y]/xy$.
Square-free discrete Hodge algebras over a field are also called Stanley--Reisner rings.

\begin{cor}\label{cor:hodge}
Let $R$ be a Dedekind domain, and let $G$ be a simply connected Chevalley--Demazure group scheme of isotropic rank $\ge 2$.
Then $G(A)=G(R)E(A)$ for any discrete Hodge algebra $A$ over $R$. In particular, if $R$ is of arithmetic type, then
$G(A)=E(A)$.
\end{cor}
\begin{proof}
This follows from~\cite[Corollary 1.5]{St-dh} and Corollary~\ref{cor:ari}.
\end{proof}

For non-simply connected Chevalley--Demazure group schemes, such as $\SO_{n}$, $n\ge 5$, we deduce the following
result; see~\S~\ref{sec:main} for the proof.

\begin{cor}\label{cor:non-sc}
Let $R$ be a locally principal ideal domain,
and let $G$ be any Chevalley--Demazure group scheme of isotropic rank $\ge 2$.
Then $G(R[x_1,\ldots,x_n])=G(R)E(R[x_1,\ldots,x_n])$ for any $n\ge 1$.
\end{cor}

Using a version of Lindel's lemma~\cite{L} and
N\'eron-Popescu desingularization~\cite{Po90}, one may extend the above results to higher-dimensional regular rings
in place of Dedekind domains. For equicharacteristic regular rings this was done earlier in~\cite{St-poly}. The following
theorem is proved in \S~\ref{sec:high}.

\begin{thm}\label{thm:high}
Let $G$ be a Chevalley--Demazure group scheme of isotropic rank $\ge 2$.
Let $R$ be a regular ring such that
every maximal localization of $R$ is either essentially smooth over a Dedekind domain with perfect residue fields,
or an unramified regular local ring.
Then $G(R[x])=G(R)E(R[x])$. Moreover, $G(A)=G(R)E(A)$ for any square-free discrete Hodge algebra $A$ over $R$.
\end{thm}

Let us mention a few other ramifications of known results yielded by Theorem~\ref{thm:main}.

Combining Corollary~\ref{cor:hodge} with the main result of~\cite{EJZK}, one concludes that $G(A)$ has Kazhdan's property (T)
for any simply connected Chevalley--Demazure group scheme $G$ of isotropic rank $\ge 2$ and any discrete Hodge
algebra $A$ over $\ZZ$; in particular, $G(\ZZ[x_1,\ldots x_n])$ has Kazhdan's property (T).

Combining Corollary~\ref{cor:hodge} with the main result of~\cite{RR}, one concludes that the congruence kernel of $G(A)$
is central in $G(A)$ for any simply connected Chevalley--Demazure group scheme $G$ of isotropic rank $\ge 2$ and any discrete Hodge
algebra $A$ over $R$, where $R$ is a Dedekind domain of arithmetic type, satisfying $2\in R^\times$ if the root
system of $G$ has components of type $C_n$ or $G_2$.


\section{A local--global principle}

Throughout this section, $R$ is any commutative ring with 1, $G$ is a Chevalley--Demazure group scheme of isotropic rank $\ge 2$, and $E$
denotes its elementary subgroup functor.

For
any $s\in R$ we denote by $R_s$ the localization of $R$ at $s$, and by
$F_s:R\to R_s$ the localization homomorphism, as well
as the induced homomorphism $G(R)\to G(R_s)$.
Similarly, for any maximal ideal $m$ of $R$ we denote by
$F_m:R\to R_m$ the localization homomorphism, as well
as the induced homomorphism $G(R)\to G(R_m)$.

We will need the following generalization of the Quillen--Suslin
 local-global principle for polynomial rings in one
variable (see~\cite[Theorem 3.1]{Sus},~\cite[Corollary 4.4]{Sus-K-O1},~\cite[Lemma 17]{PS},~\cite[Theorem 5.4]{Stepanov-elem})
to the case of several variables.

\begin{lem}\label{lem:lgp-1}
Let $R$ be any commutative ring. Fix $n\ge 1$.
If $g\in G(R[x_1,\ldots,x_n])$ satisfies $F_m(g)\in E(R_m[x_1,\ldots,x_n])$ for any maximal ideal $m$ of $R$, then
$g\in G(R)E(R[x_1,\ldots,x_n])$.
\end{lem}

The proof of Lemma~\ref{lem:lgp-1} uses the following three standard lemmas whose idea goes back to~\cite[Lemma 1]{Q}.

\begin{lem}\label{lem:inj}
Let $H$ be any affine $R$-scheme of finite type.
Fix $0\neq s\in R$, and let $F_s:H(R[z])\to H(R_s[z])$ be the localization map. For any $g(z),h(z)\in H(R[z])$
such that $h(0)=g(0)$ and $F_s(g(z))=F_s(h(z))$ there is $n\ge 0$ such that $g(s^nz)=h(s^nz)$.
\end{lem}
\begin{proof}
Since $H$ is an affine $R$-scheme of finite type, there is a closed embedding $H\to\Aff^k_R$ for some $k\ge 0$.
Hence it is enough to prove the claim for $H=\Aff^k_R$. If $k=0$, then $g(z)=g(0)=h(0)=h(z)$. If $k\ge 1$, the claim
readily reduces to the case $k=1$, that is, $g(z),h(z)\in R[z]$. Since $F_s(g(z))=F_s(h(z))$, there is $n\ge 0$
such that $s^ng(z)=s^nh(z)$. Since $g(0)=h(0)$, this implies $g(s^nz)=h(s^nz)$.
\end{proof}

\begin{lem}\label{lem:PS-15'}\cite[Theorem 5.2]{Stepanov-elem}
Fix $s\in R$, and let $F_s:G(R[z])\to G(R_s[z])$ be the localization homomorphism. For any $g(z)\in E(R_s[z],zR_s[z])$ there
exist $h(z)\in E(R[z],zR[z])$ and $k\ge 0$ such that $F_s(h(z))=g(s^kz)$.
\end{lem}
\begin{proof}
The statement is a particular case of~\cite[Theorem 5.2]{Stepanov-elem} if the root system $\Phi$ of $G$ is irreducible.
Assume that $\Phi$ has several irreducible components  $\Phi_i$. By~\cite[Exp. XXVI Prop. 6.1]{SGA3}
$G$ contains semisimple Chevalley--Demazure subgroup schemes $G_i$ of type $\Phi_i$ whose elementary subgroup
functors $E_i$ are generated by elementary root unipotents corresponding to roots in $\Phi_i$. Chevalley commutator relations
imply that $E$ is a direct product of all $E_i$. This reduces the claim to the case where $\Phi$ is irreducible.
\end{proof}

\begin{lem}\label{lem:PS-16'}
For any $g(x)\in G(R[x])$ such that $F_s(g(x))$ lies in $E(R_s[x])$, there exists
 $k\ge 0$ such that $g(ax)g(bx)^{-1}\in E(R[x])$ for all $a,b\in R$ satisfying
$a\equiv b\mod s^k$.
\end{lem}
\begin{proof}
Consider the element $f(z)=g(x(y+z))g(xy)^{-1}\in G(R[x,y,z])$.
Observe that $F_s(f(z))\in E(R_s[x,y,z])$ and $f(0)=1$. Since $F_s(g(x))\in E(R_s[x])$
and $f(0)=1$, we have $F_s(f(z))\in E(R_s[x,y,z],zR_s[x,y,z])$ (e.g. by~\cite[Lemma 4.1]{St-poly}). Now by
Lemma~\ref{lem:PS-15'} there exist $h(z)\in E(R[x,y,z],zR[x,y,z])$ and $k\ge 1$ such that
$F_s(h(z))=F_s(f(s^k z))$. By Lemma~\ref{lem:inj} there is $l\ge 1$ such that
$h(s^lz)=f(s^{l+k}z)$. Then $g(x(y+s^{l+k}z))g(xy)^{-1}$ lies in $\EE_P(R[x,y,z])$.
It remains to set $y=b$ and to choose a suitable $z$ depending on $a$.
\end{proof}

\begin{proof}[Proof of Lemma~\ref{lem:lgp-1}]
For any maximal ideal $m$ of $R$, since $F_m(g)\in E(R_m[x_1,\ldots,x_n])$, there is $s\in R\setminus m$ such that
$F_s(g)\in E(R_s[x_1,\ldots,x_n])$. Choose a finite set of elements $s=s_i\in R\setminus m_i$, $1\le i\le N$ as above,
so that $1=\sum_{i=1}^N c_is_i$ for some $c_i\in R$.
Consider $g$ as a function $g(x_1)$ of $x_1$. By Lemma~\ref{lem:PS-16'} there are $k_i\ge 1$ such that
$g(ax_1)g(bx_1)^{-1}\in E(R[x_1,\ldots,x_n])$
for any $a,b\in R[x_2,\ldots,x_n]$ satisfying $a\equiv b\pmod{s_i^{k_i}}$.
Since $s_i$ generate the unit ideal, their powers $s_i^{k_i}$ also generate the unit ideal, and we can replace $s_i$
by these powers without loss of generality. Set $a_j=\sum_{i=1}^{N-j}c_is_i$, $0\le j\le N$.
Then $a_{j+1}\equiv a_j\pmod{s_{n-j}}$, and
$$
g(x_1)=\Bigl(\,\prod\limits_{j=0}^{N-1} g(a_j x_1)g(a_{j+1}x_1)\Bigr)^{-1}g(0).
$$
Then $g(x_1)\in E(R[x_1,\ldots,x_n])g(0)$. Since $g(0)\in G(R[x_2,\ldots,x_n])$, we can proceed by induction.
\end{proof}

\begin{lem}\label{lem:lgp}
Fix $n\ge 1$. One has $G(R[x_1,\ldots,x_n])=G(R)E(R[x_1,\ldots,x_n])$ if and only if
$G(R_m[x_1,\ldots,x_n])=G(R_m)E(R_m[x_1,\ldots,x_n])$ for every maximal ideal $m$ of $R$.
\end{lem}
\begin{proof}
For the direct implication, see~\cite[Lemma 4.2]{St-poly}. To prove the converse,
it is enough to show that $g(x_1,\ldots,x_n)\in G(R[x_1,\ldots,x_n])$ such that $g(0,\ldots,0)=1$ satisfies
$g(x_1,\ldots,x_n)\in E(R[x_1,\ldots,x_n])$. For every maximal ideal $m$ of $R$, by assumption, one has
$F_m(g(x_1,\ldots,x_n)\in G(R_m)E(R_m[x_1,\ldots,x_n])$, and $g(0,\ldots,0)=1$ implies
$F_m(g(x_1,\ldots,x_n))\in E(R_m[x_1,\ldots,x_n])$. Then Lemma~\ref{lem:lgp-1} finishes the proof.
\end{proof}

\section{Proof of the main theorem}\label{sec:main}


The following result follows from stability results for non-stable $K_1$-funtors of Chevalley groups~\cite{Ste78,Plo93}.

\begin{lem}\label{lem:stability}
Let $R$ be a Noetherian ring of Krull dimension $\le 1$. If $\SL_2(R)=E_2(R)$, then
$G(R)=E(R)$ for any simply connected Chevalley--Demazure group scheme $G$ over
$R$.
\end{lem}
\begin{proof}
By~\cite[p. 102]{Bass-book} the maximal ideal spectrum of $R$ is a Noetherian topological space of dimension $\le 1$.
By~\cite[Theorem 1.4]{Ste78} this implies that $R$ satisfies the absolute stable range condition $ASR_3$, and
hence also Bass' stable range condition $SR_3$ in the sense of~\cite[p. 86]{Ste78}. Then by~\cite[Theorem 2.2]{Ste78}
(see also~\cite[Corollary 2.3]{Ste78}) suitable inclusions of $\SL_2$ into $G$ induce surjections
$\SL_2(R)/E_2(R)\to G(R)/E(R)$ for every simply connected Chevalley---Demazure group scheme $G$ corresponding to an
irreducible root system of classical type $A_n$, $n\ge 1$, $C_n$, $n\ge 2$, $D_n$, $n\ge 3$, or $B_n$, $n\ge 2$.
By~\cite[Theorem 4.1]{Ste78} and~\cite[Corollary 3]{Plo93} the same also holds for $G$ of type $G_2$, $F_4$, $E_6$, $E_7$,
and $E_8$. Consequently, $G(R)=E(R)$ for any simply connected Chevalley--Demazure group scheme $G$ over
$R$.
\end{proof}


For any commutative ring $R$ with 1, denote by $R(x)$ the localization of $R[x]$ at the set of all monic polynomials.

\begin{lem}\label{lem:R(x)}
Let $R$ be a discrete valuation ring or a local Artinian ring. Then $G(R(x))=E(R(x))$ for
any simply connected Chevalley--Demazure group scheme $G$ over $R$.
\end{lem}
\begin{proof}
 Since $R$ is a commutative Noetherian ring, by~\cite[Ch. IV, Proposition 1.2]{Lam-book06} $R(x)$
has the same Krull dimension as $R$. If $R$ is Artinian, then $R(x)$ is also Artinian, and hence a finite product of local
rings. Then $SL_n(R(x))=E_n(R(x))$ for all $n\ge 2$. If $R$ is a discrete valuation ring, then also
$SL_n(R(x))=E_n(R(x))$ for all $n\ge 2$ by~\cite[Ch. IV, Corollary 6.3]{Lam-book06}
(a corollary of~\cite[Proposition 1']{Mur66}).
Hence by Lemma~\ref{lem:stability} one has $G(R(x))=E(R(x))$ in both cases.
\end{proof}

We will also use the following lemma, that was established in~\cite[Corollary 5.7]{Sus} for $G=\GL_n$.

\begin{lem}\cite[Lemma 2.7]{St-serr}\label{lem:inj-f}
Let $A$ be a commutative ring, and let $G$ be a reductive group scheme over $A$, such that every semisimple normal subgroup
of $G$ is isotropic. Assume moreover that for any maximal ideal $m\subseteq A$,
 every semisimple normal subgroup of $G_{A_m}$ contains $(\Gm_{,A_m})^2$.
Then for any monic polynomial $f\in A[x]$ the natural homomorphism
$$
G(A[x])/E(A[x])\to G(A[x]_f)/E(A[x]_f)
$$
is injective.
\end{lem}

Now we are ready to establish the main theorem for simply connected semisimple Chevalley--Demazure group schemes.

\begin{thm}\label{thm:main-2}
Let $R$ be a locally principal ideal ring or Artinian ring. Then
$$
G(R[x_1,\ldots,x_n])=G(R)E(R[x_1,\ldots,x_n])
$$ for
any simply connected Chevalley--Demazure group scheme $G$ over $R$ of isotropic rank $\ge 2$ and any $n\ge 1$.
\end{thm}
\begin{proof}
For every maximal ideal $m$ of $R$, the ring $R_m$ is
a local principal ideal domain, i.e. a discrete valuation ring, or a local Artinian ring,
In both cases $R_m$ is a local Noetherian ring of Krull dimension $\le 1$.
By~\cite[Ch. IV, Proposition 1.2]{Lam-book06} $R_m(x_1)$
has the same Krull dimension as $R_m$. If $R_m$ is Artinian, then $R_m(x_1)$ is also Artinian. If $R_m$
is a discrete valuation ring, then $R_m(x_1)$ is a principal ideal domain by~\cite[Ch. IV, Corollary 1.3]{Lam-book06}.
Hence by induction hypothesis
$$
G\bigl(R_m(x_1)[x_2,\ldots,x_n]\bigr)=G\bigl(R_m(x_1)\bigr)E\bigl(R_m(x_1)[x_2,\ldots,x_n]\bigr).
$$
 Then by Lemma~\ref{lem:R(x)}
$G\bigl(R_m(x_1)[x_2,\ldots,x_n]\bigr)=E\bigl(R_m(x_1)[x_2,\ldots,x_n]\bigr)$. Then by Lemma~\ref{lem:inj-f} we have
$G\bigl(R_m[x_1,\ldots,x_n]\bigr)=E\bigl(R_m[x_1,\ldots,x_n]\bigr)$. Then
Lemma~\ref{lem:lgp} finishes the proof.
\end{proof}

To pass from simply connected Chevalley--Demazure group schemes to general ones, we use the following reduction lemma.

\begin{lem}\label{lem:sc-red}
Let $G$ be a Chevalley--Demazure group scheme, and let $E$ be an elementary subgroup functor of $G$.
Let $G^{\scl}$ be the simply connected cover of the adjoint semisimple group scheme $G^{\ad}=G/\Cent(G)$,
and let $E^{\scl}$ be its elementary subgroup functor corresponding to the pinning compatible with that of $G$. Let $A$ be a normal Noetherian integral domain.
If one has
$
G^{\scl}(A[x])=G^{\scl}(A)E^{\scl}(A[x]),
$
then
$
G(A[x])=G(A)E(A[x]).
$
\end{lem}
\begin{proof}
There is a short exact sequence of $\ZZ$-group schemes
$$
1\to [G,G]\to G\to T\to 1,
$$
for a split $\ZZ$-torus $T$. Here the group $[G,G]$ is the algebraic derived subgroup scheme of $G$ in the sense of~\cite[Exp. XXII, \S 6.2]{SGA3}. It is a semisimple Chevalley--Demazure
group scheme, and $E(A)\le [G,G](A)$. Since
$T(A[x])=T(A)$,
the exact sequence
$$
1\to [G,G](A[x])\to G(A[x])\to T(A[x])
$$
implies that it is enough to prove the claim for $[G,G]$. In other words, we may assume that $G$ is semisimple. Then there is a short exact sequence of algebraic groups
$$
1\to C\xrightarrow{i} G^{\scl}\xrightarrow{\pi} G\to 1,
$$
where $C$ is a group of multiplicative type over $\ZZ$, central in $G^{\scl}$. Write the respective ``long'' exact sequences over $A[x]$ and $A$
with respect to fppf topology. Adding the maps induced by the homomorphism $\rho:A[x]\to A$, $x\mapsto 0$,
we obtain a commutative diagram
\begin{equation*}
\xymatrix@R=15pt@C=20pt{
1\ar[r]\ar@{=}[d]&C(A[x])\ar[d]^\rho\ar[r]^{i}&G^{\scl}(A[x])\ar[d]^\rho\ar[r]^\pi &G(A[x])\ar[d]^\rho\ar[r]^{\hspace{-15pt}\delta}&H^1_{\fppf}(A[x],C)\ar[d]^{\cong}\\
1\ar[r]&C(A)\ar[r]^{i}&G^{\scl}(A)\ar[r]^\pi&G(A)\ar[r]^{\hspace{-15pt}\delta}&H^1_{\fppf}(A,C)\\
}
\end{equation*}
Here the rightmost vertical arrow is an isomorphism by~\cite[Lemma 2.4]{CTS}. Take any
$g\in\ker\bigl(\rho:G(A[x])\to G(A)\bigr)$.
It is enough to show that $g\in E(A[x])$.

We have $\delta(g)=1$, hence there is $\tilde g\in G^{\scl}(A[x])$ with
$\pi(\tilde g)=g$. Clearly, $\rho(\tilde g)\in C(A)$, and hence
$$
\tilde g\in C(A)\cdot \ker\bigl(\rho:G^{\scl}(A[x])\to G^{\scl}(A)\bigr)\subseteq C(A)E^{\scl}(A[x]).
$$
Since
$\pi(E^{\scl}(A[x]))=E(A[x])$, this proves the claim.
\end{proof}

\begin{proof}[Proof of Corollary~\ref{cor:non-sc}]
By Lemma~\ref{lem:lgp} it is enough to prove the claim for $R_m$, where $m$ is any maximal ideal of $R$. Since $R_m$
is a discrete valuation ring, the claim follows from Lemma~\ref{lem:sc-red} and Theorem~\ref{thm:main}.
\end{proof}

\section{Extension to higher dimensional regular rings}\label{sec:high}

In this section we discuss extensions of Theorem~\ref{thm:main} to rings of polynomials over
higher dimensional regular rings $R$. Note that the following result is contained in~\cite{St-poly}.

\begin{thm}\label{thm:equi}
Let $G$ be a Chevalley--Demazure group scheme of isotropic rank $\ge 2$. Let $R$ be
an equicharacteristic regular domain. Then $G(R[x_1,\ldots,x_n])=G(R)E(R[x_1,\ldots,x_n])$ for any $n\ge 1$.
\end{thm}
\begin{proof}
The claim follows from~\cite[Theorem 1.3]{St-poly}, since
$R[x_1,\ldots,x_n]$ is a regular domain containig a perfect field, for any $n\ge 1$.
\end{proof}

Thus, it remains to consider the case of regular domains $R$ of unequal characteristic.
Following~\cite{W10}, we rely on the following generalization of Lindel's lemma~\cite{L}. See also~\cite[Proposition 2.1]{Pop89}
for a slightly weaker version.

\begin{lem}\cite[Theorem 1.3]{Dut}\label{lem:dut}
Let $(A,m)$ be a regular local ring of dimension $d + 1$,
essentially of finite type and smooth over an excellent discrete valuation ring $(V,(\pi))$ such that $K=A/m$ is
separably generated over $V/\pi V$. Let $0\neq a\in m^2$ be such that $a\not\in \pi A$.
Then there exists a regular local subring $(B,n)$ of $(A,m)$, with $B/n=A/m=K$, and such that
\begin{enumerate}
\item  $B$ is a localization of a polynomial ring $W[x_1,...,x_d]$ at a maximal ideal
of the type $(\pi,f(x_1 ),x_2,\ldots,x_d )$ where $f$ is a monic irreducible polynomial
in $W[x_1]$ and $(W,(\pi))$ is an excellent discrete valuation ring contained in $A$; moreover $A$ is an \'etale
neighborhood of $B$.
\item There exists an element $h\in B\cap aA$ such that $B/hB\cong A/aA$ is an isomorphism. Furthermore $hA = aA$.
\end{enumerate}
\end{lem}

\begin{lem}\label{lem:perfect-res}
Let $G$ be a Chevalley--Demazure group scheme $G$ of isotropic rank $\ge 2$.
Let $R$ be a Dedekind domain with perfect residue fields.
Let $A$ be a regular $R$-algebra that is essentially smooth over $R$.
Then $G(A[x])=G(A)E(A[x])$.
\end{lem}
\begin{proof}
By Lemma~\ref{lem:lgp} we can assume that $A$ is local.
Then, in particular, $R$ is a regular domain, and hence we can assume that $G$ is simply connected by Lemma~\ref{lem:sc-red}.
The map $R\to A$ factors through a localization $R_q$,
for a prime ideal $q$ of $R$. If $R_q$ is equicharacteristic, we are done
by Theorem~\ref{thm:equi}. Otherwise $R_q=V$ is a discrete valuation ring of characteristic $0$,
and hence excellent by~\cite[Scholie 7.8.3]{EGAIV-2}.
The residue field of $A$ is a finitely generated field extension of the perfect field $R_q/qR_q=V/\pi$,
hence it is separably generated. Thus, all conditions of Lemma~\ref{lem:dut} are fulfilled.

The rest of the proof
proceeds as the proof of~\cite[Lemma 6.3]{St-poly}, using
Lemma~\ref{lem:dut} instead of Lindel's lemma, and
 Theorem~\ref{thm:main-2} instead of~\cite[Theorem 1.2]{St-poly}.
Namely, one proceeds by induction on $\dim A=d+1$. If $\dim A=1$,
we are in the setting of Theorem~\ref{thm:main-2}. Assume $\dim A\ge 2$. Then $m^2\setminus \pi A$ is non-empty, since
$A/\pi A$ is an essentially smooth, hence regular, local ring over  $V/\pi$, hence a domain. For any $a\in m^2\setminus \pi A$,
let $B$ and $h\in B\cap aA$ be as in Lemma~\ref{lem:dut}. Since $B$ is a localization of a polynomial
ring over a discrete valuation ring, which is subject to Theorem~\ref{thm:main-2},
by Lemma~\ref{lem:lgp} one has $G(B[x])=G(B)E(B[x])$.
We need to show that any element $g(x)\in G(A[x])$ belongs to $G(A)E(A[x])$.
Since $\dim A_h<\dim A$, the element
$g(x)$ belongs to $G(A_h)E(A_h[x])$. Clearly, we can assume from the start that
$g(0)=1$, then in fact $g(x)\in E(A_h[x])$.
By Lemma~\ref{lem:dut} $h$ satisfies $Ah+B=A$, $Ah\cap B=Bh$.
Hence by~\cite[Lemma 3.4 (i)]{St-poly} we have
$g(x)=g_1(x)g_2(x)$
for some $g_1(x)\in E(A[x])$ and $g_2(x)\in G(B_h[x])$.
Then $g_2(x)\in G(B_h[x])\cap G(A[x])=G(B[x])$.
Then we have $g_2(x)\in G(B)E(B[x])$. Therefore,
$g(x)\in G(A)E(A[x])$.

\end{proof}

\begin{lem}\label{lem:unram}
Let $G$ be a Chevalley--Demazure group scheme of isotropic rank $\ge 2$.
Let $R$ be a regular ring such that every maximal localization of $R$ is an unramified regular local ring.
Then $G(R[x])=G(R)E(R[x])$.
\end{lem}
\begin{proof}
By Lemma~\ref{lem:lgp} we can assume that $R$ is an unramified regular local ring with maximal ideal $m$.
If $R$ is equicharacteristic, we are done
by Theorem~\ref{thm:equi}. If $R$ has characteristic $0$ and residual characteristic $p$, then by assumption $p\not\in m^2$.
Then $R$ is geometrically regular over $\ZZ_{(p)}$~\cite[p. 4]{Swan}, and hence a filtered inductive limit of regular local rings
which are essentially smooth over
$\ZZ_{(p)}$ by~\cite[Corollary 1.3]{Swan}. Then Lemma~\ref{lem:perfect-res} finishes the proof.
\end{proof}

\begin{proof}[Proof of Theorem~\ref{thm:high}]
For the first claim, combine Lemma~\ref{lem:lgp}, Lemma~\ref{lem:perfect-res}, and Lemma~\ref{lem:unram}.
For the second claim, add~\cite[Theorem 1.3]{St-dh}.
\end{proof}

\renewcommand{\refname}{References}

\end{document}